\documentclass[12pt, a4paper, reqno]{amsart}

\usepackage{amssymb,mathtools,cite,enumerate,color,eqnarray,hyperref}

\definecolor{blue}{RGB}{0, 0, 200}
\definecolor{pink}{RGB}{252, 0, 50}

\hypersetup{
           colorlinks=true,
           linkcolor=blue,
           urlcolor=red,
           citecolor=pink,
          }

\newtheorem{theorem}{Theorem}[section]
\newtheorem{corollary}[theorem]{Corollary}

\newtheorem{lemma}[theorem]{Lemma}

\theoremstyle{definition}
\newtheorem{definition}[theorem]{Definition}

\numberwithin{equation}{section}

\begin{document}
\title[The pod function and its connection with other partition functions]{The pod function and its connection with other partition functions}
\author[Hemjyoti Nath]{Hemjyoti Nath}

\maketitle

\let\thefootnote\relax
\footnotetext{$2020$ \textit{Mathematics subject classification.} 11P81, 11P83.}

\begin{abstract}
The number of partitions of $n$ wherein odd parts are distinct and even parts are unrestricted, often denoted by $pod(n)$. In this paper, we provide linear recurrence relations for $pod(n)$, and the connections of $pod(n)$ with other partition functions.
\end{abstract}

\section{Introduction}
A partition of a positive integer $n$ is a non-increasing sequence of positive integers, called parts, whose sum equals $n$. For example, $n=4$ has five partitions, namely, 
\begin{equation*}
    4, \quad 3 + 1, \quad 2+2, \quad 2+1+1, \quad 1 + 1 + 1 + 1.
\end{equation*}
If $p(n)$ denotes the number of partitions of $n$, then $p(4)=5$. The generating function for $p(n)$ is given by
\begin{equation*}
    \sum_{n=0}^{\infty}p(n)q^n = \frac{1}{(q;q)_{\infty}},
\end{equation*}
where, here and throughout the paper
\begin{equation*}
    (a;q)_{\infty} = \prod_{n=0}^{\infty}(1-aq^n), \quad |q|<1.
\end{equation*}

In $2010$, Hirschhorn and Sellers \cite{9} defined the partition function $pod(n)$, which counts the number of partitions of $n$ into distinct odd parts where even parts maybe repeated. For example, $pod(4)=3$ with the relevent partitions being
\begin{equation*}
    4, \quad 3+1, \quad 2+2.
\end{equation*}
The generating function for $pod(n)$ is given by 
\begin{equation*}
    \sum_{n=0}^{\infty}pod(n)q^n = \frac{(-q;q^2)_{\infty}}{(q^2;q^2)_{\infty}} = \frac{(q^2;q^4)_{\infty}}{(q;q)_{\infty}} = \frac{(q^2;q^2)_{\infty}}{(q;q)_{\infty}(q^4;q^4)_{\infty}}.
\end{equation*}
Hirschhorn and Sellers proved the Ramanujan-type congruences
\begin{equation*}
    pod\left(3^{2\alpha+3}n + \frac{23\times 3^{2\alpha+2}+1}{8}\right) \equiv 0 \pmod{3}, \quad \text{for all} \quad \alpha\geq0 , \hspace{2mm} n\geq 0
\end{equation*}
using some $q$-series identities. Radu and Sellers \cite{10} later found congruences for $pod(n)$ modulo $5$ and $7$ using the theory of modular forms.

For nonnegative integers $n$ and $k$, let $r_k(n)$ (resp. $t_k(n)$) denote the number of representations of $n$ as sum of $k$ squares (resp. triangular numbers). In 2011, based on the generating function of $pod(3n+2)$ found in \cite{9}, Lovejoy and Osburn discovered the following arithmetic relation
\begin{equation*}
    pod(3n + 2) \equiv (-1)^n r_5(8n + 5) \pmod{3}.
\end{equation*}
Recently, Ballantine and Merca \cite{13}, obtained new properties for $pod(n)$ using the connections with $4$-regular partitions and, for fixed $k \in \{0,2\}$, partitions into distinct parts not congruent to $k$ modulo $4$.

In a very recent work, Ballantine and Welch \cite{12} proved a recurrence relation for $pod(n)$ combinatorially.

Motivated from their work, we establish new recurrence relations for $pod(n)$ that involve triangular numbers and generalized pentagonal numbers i.e.,
\begin{equation*}
    T_k=\frac{k(k+1)}{2}, \quad k\in \mathbb{N}_0
\end{equation*}
and
\begin{equation*}
    G_k = \frac{1}{2} \left\lceil \frac{k}{2} \right\rceil \left\lceil \frac{3k+1}{2} \right\rceil, \quad k\in \mathbb{N}_0
\end{equation*}
respectively, and its connections with other partition functions.

\section{Preliminaries}
We require the following definitions, lemmas and theorems to prove the main results in the next two sections. The Jacobi triple product identity \cite{9},  is given by
\begin{equation}\label{e0.1}
    (q^2;q^2)_{\infty}(-qx;q^2)_{\infty}(-q/x;q^2)_{\infty} =\sum_{n=-\infty}^{\infty}x^nq^{n^2}, \qquad |q|<1, \quad x \neq 0.
\end{equation}
For $|qx| <1$, Ramanujan's general theta function $f(q,x)$ is defined as 
\begin{equation}\label{e2.0.0}
    f(q,x) : = \sum_{n=-\infty}^{\infty}q^{n(n+1)/2}x^{n(n-1)/2}.
\end{equation}
Using $\eqref{e0.1}$, $\eqref{e2.0.0}$ takes the shape
\begin{equation}\label{e2.0.1}
    f(q,x)=(-q,qx)_{\infty}(-x,qx)_{\infty}(qx,qx)_{\infty}.
\end{equation}
The special cases of $f(q,x)$ are
\begin{equation}\label{e2.0.2}
    \varphi(q):=f(q,q)=\sum_{n=-\infty}^{\infty}q^{n^2}=(-q;q^2)_{\infty}^2(q^2;q^2)_{\infty}=\frac{(q^2;q^2)_{\infty}^5}{(q;q)_{\infty}^2(q^4;q^4)_{\infty}^2},
\end{equation}
\begin{equation}\label{e2.0.3}
    \psi(q):=f(q,q^3)=\sum_{n=0}^{\infty}q^{n(n+1)/2}=\frac{(q^2;q^2)_{\infty}}{(q;q^2)_{\infty}} = \frac{(q^2;q^2)_{\infty}^2}{(q;q)_{\infty}},
\end{equation}
\begin{equation}\label{e2.0.4}
    \varphi(-q):=f(-q,-q)=\sum_{n=-\infty}^{\infty}(-1)^nq^{n^2}=\frac{(q;q)_{\infty}^2}{(q^2;q^2)_{\infty}},
\end{equation}
\begin{equation}\label{e2.0.3.1}
    \psi(-q):=f(-q,-q^3)=\sum_{n=0}^{\infty}(-1)^{n(n+1)/2}q^{n(n+1)/2}=\frac{(q^2;q^2)_{\infty}}{(-q;q^2)_{\infty}} = \frac{(q^2;q^2)_{\infty}^2}{(-q;-q)_{\infty}}.
\end{equation}
\begin{theorem}[Euler's Pentagonal Number Theorem \cite{9}]
We have
    \begin{equation}\label{e2.0.3.2}
    \sum_{n=-\infty}^{\infty}(-1)^n q^{n(3n+1)/2}=(q;q)_{\infty}.
\end{equation}
\end{theorem}
\begin{theorem}[Jacobi's Identity \cite{9}]
We have
    \begin{equation}\label{e2.0.3.3}
    \sum_{n=0}^{\infty}(-1)^n(2n+1)q^{n(n+1)/2}=(q;q)_{\infty}^3.
\end{equation}
\end{theorem}

\begin{lemma}[Hirschhorn \cite{2}]
    We have the following identities due to Ramanujan,
 \begin{equation}\label{e2.1.1}
         \frac{(q;q)_{\infty}^5}{(q^2;q^2)_{\infty}^2} = \sum_{n=-\infty}^{\infty}(6n+1)q^{n(3n+1)/2},
\end{equation}
\begin{equation}\label{e2.1.2}
        \frac{(q;q)_{\infty}^2(q^4;q^4)_{\infty}^2}{(q^2;q^2)_{\infty}} = \sum_{n=-\infty}^{\infty}(3n+1)q^{3n^2+2n}.
    \end{equation}
\end{lemma}

\begin{lemma}[Baruah \cite{3}]
We have
 \begin{equation}\label{e2.1.3}
    \sum_{n=-\infty}^{\infty}q^{n(3n+1)/2}=\frac{(-q;q)_{\infty}(q^3;q^3)_{\infty}}{(-q^3;q^3)_{\infty}}.
\end{equation}
\end{lemma}

\section{Recurrence relations}
In this section, we prove some recurrence relations for the partition function $pod(n)$ that involve $T_k$ and $G_k$ for $k\in \mathbb{N}_0$.

\begin{theorem}
    For $n\geq 0$, we have
    \begin{equation*}
        \sum_{j=0}^{\infty}(-1)^{\lceil j/2 \rceil}pod(n-T_j)=\begin{cases}
            1, \quad if \hspace{1mm} n=0, \\
            0, \quad otherwise.
        \end{cases}
    \end{equation*}
\end{theorem}

\begin{proof}
Replacing $q$ by $q^2$ and $x$ by $-q$ in $\eqref{e0.1}$, we have the relation
    \begin{equation}\label{e1}
        (q^4;q^4)_{\infty}(q^3;q^4)_{\infty}(q;q^4)_{\infty}= \sum_{n=-\infty}^{\infty}(-1)^nq^{2n^2+n},
            \end{equation}
Multiplying the above equation by $(q^2;q^2)_{\infty}/(q;q)_{\infty}(q^4;q^4)_{\infty}$, we get
\begin{equation}\label{e1.1.1}
   1=\frac{(q^2;q^2)_{\infty}}{(q;q)_{\infty}(q^4;q^4)_{\infty}}\sum_{n=0}^{\infty}(-1)^{\lceil n/2 \rceil}q^{n(n+1)/2}.
\end{equation}
This can be written as
\begin{equation*}
    \left( \sum_{n=0}^{\infty}pod(n)q^n \right)\left(\sum_{n=0}^{\infty}(-1)^{\lceil n/2 \rceil}q^{n(n+1)/2}\right)=1.
\end{equation*}
The proof follows easily applying the well known Cauchy multiplication of two power series on the left hand side of the above equation.
\end{proof}

\begin{corollary}
    For $n\geq 0$, we have
    \begin{equation*}
       \sum_{j=0}^{\infty}pod(n-T_j) \equiv  \begin{cases}
            1 \pmod{2}, \quad \text{if} \hspace{2.5mm} n=0,\\
            0 \pmod{2}, \quad \text{otherwise}.
        \end{cases}
        \end{equation*}
\end{corollary}

The proof is immediate from the above theorem.

\begin{theorem}
    For $n\geq 0$, we have
    \begin{equation*}
        \sum_{j=0}^{\infty}(-1)^j (2j+1) pod(n-2T_j)=\begin{cases}
            \displaystyle\sum_{j=0}^{\infty}(-1)^{\lceil k/2 \rceil + \lceil j/2 \rceil + G_k - 2G_j}, \quad if \hspace{1mm} n=G_k, \\
            0, \quad otherwise.
        \end{cases}
    \end{equation*}
\end{theorem}

\begin{proof}
    We can write the generating function of $pod(n)$ as
    \begin{equation*}
        \sum_{n=0}^{\infty}pod(n)q^n = \frac{(-q;-q)_{\infty}}{(q^2;q^2)_{\infty}^2}.
    \end{equation*}
Multiplying both sides by $(q^2;q^2)_{\infty}^3$ in the above equation, we arrive at
    \begin{equation*}
        (q^2;q^2)_{\infty}^3 \sum_{n=0}^{\infty}pod(n)q^n  = (q^2;q^2)_{\infty} (-q;-q)_{\infty}
    \end{equation*}
Using $\eqref{e2.0.3.2}$ and $\eqref{e2.0.3.3}$, we get
    \begin{equation*}      
    \left( \sum_{n=0}^{\infty}(-1)^n(2n+1)q^{T_n} \right)\left( \sum_{n=0}^{\infty}pod(n)q^{n} \right)  = \left( \sum_{k=0}^{\infty}(-1)^{\lceil k/2 \rceil} q^{2G_k} \right) \left( \sum_{k=0}^{\infty} (-1)^{\lceil k/2 \rceil + G_k } q^{G_k} \right).
    \end{equation*}
Finally, equating the coefficient of $q^n$ on each side gives the result.
\end{proof}

\begin{theorem}
    For $n\geq 0$,
    \begin{equation*}
        \sum_{j=-\infty}^{\infty} (3j+1) pod(n-3j^2-2j)=\begin{cases}
             \displaystyle\sum_{j=0}^{\infty}(-1)^{\lceil k/2 \rceil + \lceil j/2 \rceil - 4G_j}, \quad if \hspace{1mm} n=G_k, \\
            0, \quad otherwise.
        \end{cases}
    \end{equation*}
\end{theorem}

\begin{proof}
From $\eqref{e2.1.2}$, we have
    \begin{equation*}
       \frac{(q;q)_{\infty}^2(q^4;q^4)_{\infty}^2}{(q^2;q^2)_{\infty}} = \sum_{n=-\infty}^{\infty}(3n+1)q^{3n^2+2n},
    \end{equation*}
Multiplying the above equation by $(q^2;q^2)_{\infty}/(q;q)_{\infty}(q^4;q^4)_{\infty}$, to get
\begin{equation*}
\frac{(q^2;q^2)_{\infty}}{(q;q)_{\infty}(q^4;q^4)_{\infty}}\left( \sum_{n=-\infty}^{\infty}(3n+1)q^{3n^2+2n} \right)  = (q;q)_{\infty}(q^4;q^4)_{\infty}
\end{equation*}
On simplification, we get
\begin{equation}\label{e1.1}
\left( \sum_{n=0}^{\infty}pod(n)q^n \right)\left( \sum_{n=-\infty}^{\infty}(3n+1)q^{3n^2+2n} \right)  = \left( \sum_{k=0}^{\infty}(-1)^{\lceil k/2 \rceil} q^{G_k} \right) \left( \sum_{k=0}^{\infty}(-1)^{\lceil k/2 \rceil}  q^{4G_k}\right). 
\end{equation}
Finally, equating the coefficient of $q^n$ on each side gives the result.
\end{proof}

\begin{corollary}
 For $n\geq 0$, we have
    \begin{equation*}
        \sum_{j=-\infty}^{\infty} (3j+1) pod(n-3j^2-2j)=\begin{cases}
            \displaystyle\sum_{j=0}^{\infty}(-1)^{\lceil k/2 \rceil + \lceil j/2 \rceil - 2G_j}, \quad if \hspace{1mm} n=T_k, \\
            0, \quad otherwise.
        \end{cases}
    \end{equation*}    
\end{corollary}

\begin{proof}
From $\eqref{e1.1.1}$, we have
    \begin{equation*}
        (q;q)_{\infty}(q^4;q^4)_{\infty}=(q^2;q^2)_{\infty}\sum_{n=0}^{\infty}(-1)^{\lceil n/2 \rceil}q^{n(n+1)/2}.
    \end{equation*}
And, from $\eqref{e1.1}$ we find that
\begin{equation*}
     \left( \sum_{n=0}^{\infty}pod(n)q^n \right)\left( \sum_{n=-\infty}^{\infty}(3n+1)q^{3n^2+2n} \right)  = (q^2;q^2)_{\infty}\sum_{k=0}^{\infty}(-1)^{\lceil k/2 \rceil}q^{k(k+1)/2},
\end{equation*}
which is equivalent to
\begin{equation*}
     \left( \sum_{n=0}^{\infty}pod(n)q^n \right)\left( \sum_{n=-\infty}^{\infty}(3n+1)q^{3n^2+2n} \right)  = \left( \sum_{k=0}^{\infty}(-1)^{\lceil k/2 \rceil} q^{2G_k} \right) \left( \sum_{k=0}^{\infty}(-1)^{\lceil k/2 \rceil}q^{T_k} \right).
\end{equation*}
Finally, equating the coefficient of $q^n$ on each side gives the result.
\end{proof}

\section{Connections of $pod(n)$ with other partition functions}
In this section, we will see how $pod(n)$ is related to various other partition functions.
\subsection{Connection with the partition function $ped(n)$}
The partition function $ped(n)$ enumerates the number of partitions of $n$ where the even parts are distinct and odd parts maybe repeated. For example, $ped(6)=9$ and the partitions are 
\begin{equation*}
    6, \quad 5+1, \quad 4+2, \quad 4+1+1, \quad 3+3, \quad 3+2+1, \quad 3+1+1+1, \quad 2+1+1+1+1, \quad 1+1+1+1+1+1.
\end{equation*}
The generating function for $ped(n)$ is given by
\begin{equation}\label{e2.1}
    \sum_{n=0}^{\infty}ped(n)q^n = \frac{(-q^2;q^2)_{\infty}}{(q;q^2)_{\infty}}=\frac{(q^4;q^4)_{\infty}}{(q;q)_{\infty}}.
\end{equation}

In this section, we shall prove that the partition function $pod(n)$ can be expressed in terms of the partition function $ped(n)$.

\begin{theorem}
For any nonnegative integer $n$, we have
    \begin{equation*}
        pod(n) = \sum_{j=0}^{\infty}(-1)^jpod\left( \frac{j}{2} \right)ped(n-j).
    \end{equation*}
with $pod(x)=0$ if $x$ is not an integer.
\end{theorem}

\begin{proof}
It is known that 
\begin{equation*}
    \sum_{n=0}^{\infty}pod(n)q^n = \frac{1}{\psi(-q)}.
\end{equation*}
We can write the generating function of $pod(n)$ as
   \begin{equation*}
       \sum_{n=0}^{\infty} pod(n)q^n  = \frac{(q^2;q^2)_{\infty}}{(q;q)_{\infty}(q^4;q^4)_{\infty}},
    \end{equation*}
which after some manipulations takes the shape
    \begin{equation*}
     \sum_{n=0}^{\infty} pod(n)q^n = \frac{1}{\psi(q^2)} \sum_{n=0}^{\infty}ped(n)q^n,
    \end{equation*}
which is equivalent to
    \begin{equation*}
     \sum_{n=0}^{\infty} pod(n)q^n     = \left( \sum_{n=0}^{\infty}(-1)^{n}pod(n)q^{2n} \right) \left( \sum_{n=0}^{\infty}ped(n)q^n \right).
   \end{equation*}
Finally, equating the coefficient of $q^n$ on each side gives the result.
\end{proof}

\subsection{Connection with the partitions into distinct odd parts}

The number of partitions of $n$ into distinct parts is usually denoted by $q(n)$. The number of partitions of $n$ into distinct odd parts is denoted by $q_{odd}(n)$. The generating functions for $q(n)$ and $q_{odd}(n)$ are given by

\begin{equation*}
    \sum_{n=0}^{\infty}q(n)q^n=(-q;q)_{\infty},
\end{equation*}
and
\begin{equation*}
    \sum_{n=0}^{\infty}q_{odd}(n)q^n=(-q;q^2)_{\infty}.
\end{equation*}

\begin{theorem}
    For any nonnegative integer $n$, the partition functions $p(n)$, $q_{odd}(n)$ and $pod(n)$ are related by
    \begin{equation*}
        pod(n)=\sum_{j=0}^{\infty}q_{odd}(j)p\left( \frac{n}{2}-\frac{j}{2} \right),
    \end{equation*}
with $p(x)=0$ if $x$ is not an integer.
\end{theorem}

\begin{proof}
 Considering the generating function of $pod(n)$, we can write
\begin{equation*}
    \sum_{n=0}^{\infty}pod(n)q^n = \frac{(-q;q^2)_{\infty}}{(q^2;q^2)_{\infty}}
\end{equation*}
which is equivalent to
\begin{equation*}
    \sum_{n=0}^{\infty}pod(n)q^n  = \left( \sum_{n=0}^{\infty}p(n)q^{2n} \right)\left( \sum_{n=0}^{\infty}q_{odd}q^{n} \right).
\end{equation*}
Finally, equating the coefficient of $q^n$ on each side gives the result.
\end{proof}

\begin{theorem}
    For any nonnegative integer $n$, the partition functions $q_{odd}(n)$ and $pod(n)$ are related by

\begin{equation*}
        q_{\text{odd}}(n) = \sum_{j=-\infty}^{\infty}(-1)^j pod\left( \frac{n}{2}-\frac{j(3j+1)}{2} \right).
\end{equation*}
with $pod(x)=0$ if $x$ is not an integer. 
\end{theorem}

\begin{proof}
We can write
    \begin{equation*}
    {(q^2;q^2)_{\infty}}\sum_{n=0}^{\infty}pod(n)q^n  = {(-q;q^2)_{\infty}},
    \end{equation*}
which is equivalent to
    \begin{equation*}
    \left( \sum_{j=-\infty}^{\infty} (-1)^j q^{j(3j+1)} \right)\left( \sum_{n=0}^{\infty}pod(n)q^n \right)  = \sum_{n=0}^{\infty}q_{odd}(n)q^n.
\end{equation*}
Finally, equating the coefficient of $q^n$ on each side gives the result.
\end{proof}

We denote the difference of the number of partitions of $n$ into an even number of parts and partitions of $n$ into an odd number of parts by $p_{e-o}(n)$. The generating function of $p_{e-o}(n)$ is given by 
\begin{equation*}
    \sum_{n=0}^{\infty}p_{e-o}(n)q^n = \frac{1}{(-q;q)_{\infty}}.
\end{equation*}

We have the following result.

\begin{corollary}
    For any nonnegative integer $n$, the partition functions $ p_{\text{e$-$o}}(n)$ and $pod(n)$ are related by
    \begin{equation*}
    p_{e-o}(n) = (-1)^n\sum_{j=-\infty}^{\infty}(-1)^j pod\left( \frac{n}{2}-\frac{j(3j+1)}{2} \right).
\end{equation*}
with $pod(x)=0$ if $x$ is not an integer.
\end{corollary}
    The proof is immediate as
    \begin{equation*}
         q_{odd}(n)=(-1)^n p_{e-o}(n).
    \end{equation*}

\subsection{Connection with the overpartition and the partition function $A(n)$}
An overpartition of a nonnegative integer $n$ is a non-increasing sequence of natural numbers whose sum is $n$, where the first occurence (or equivalently, the last occurence) of a number may be overlined. The eight overpartitions of $3$ are
\begin{equation*}
3,\quad \overline{3} ,\quad 2+1,\quad \overline{2}+1,\quad 2+\overline{1} ,\quad \overline{2}+\overline{1} ,\quad 1+1+1, \quad \overline{1}+1+1.
\end{equation*}
The number of overpartitions of $n$ is denoted by $\overline{p}(n)$ and its generating function is given by
\begin{equation*}
    \sum_{n=0}^{\infty}\overline{p}(n)q^n = \frac{(-q;q)_{\infty}}{(q;q)_{\infty}} = \frac{(q^2;q^2)_{\infty}}{(q;q)_{\infty}^2}.
\end{equation*}
Recently, Merca \cite{8} defined the following functions.
\begin{definition}
For a positive integer $n$, let
\begin{enumerate}
    \item $a_e(n)$ be the number of partitions of $n$ into an even number of parts in which the even parts can appear in two colours.
    \item $a_o(n)$ be the number of partitions of $n$ into an odd number of parts in which the even parts can appear in two colours.
    \item $A(n)=a_e(n)-a_o(n)$.
\end{enumerate}
\end{definition}
He found that the generating function of $A(n)$ is
\begin{equation*}
    \sum_{n=0}^{\infty}A(n)q^n = \frac{1}{(-q;q)_{\infty}(-q^2;q^2)_{\infty}}=(q;q^2)_{\infty}(q^2;q^4)_{\infty} = \frac{(q;q)_{\infty}}{(q^4;q^4)_{\infty}}.
\end{equation*}

\begin{theorem}
    For any nonnegative integer $n$, we have
    \begin{equation*}
        pod(n) = \sum_{j=0}^{\infty}A(j)\overline{p}(n-j).
    \end{equation*}
\end{theorem}

\begin{proof}
We start with the generating function of $\overline{p}(n)$,
    \begin{equation*}
        \sum_{n=0}^{\infty}\overline{p}(n)q^n = \frac{(q^2;q^2)_{\infty}}{(q;q)_{\infty}^2}.
    \end{equation*}
Multiplying the above equation by $(q;q)_{\infty}/(q^4;q^4)_{\infty}$, to get
    \begin{equation*}
        \frac{(q;q)_{\infty}}{(q^4;q^4)_{\infty}}  \sum_{n=0}^{\infty}\overline{p}(n)q^n = \frac{(q^2;q^2)_{\infty}}{(q;q)_{\infty}(q^4;q^4)_{\infty}},
    \end{equation*}
which is equivalent to
    \begin{equation*}
        \left( \sum_{n=0}^{\infty}A(n)q^n \right)\left( \sum_{n=0}^{\infty}\overline{p}(n)q^n \right) =\sum_{n=0}^{\infty}pod(n)q^n.
    \end{equation*}
Finally, equating the coefficient of $q^n$ on each side gives the result.
\end{proof}

\subsection{Connection with the overpartitions into odd parts}
We denote by $\overline{p}_o(n)$ the number of overpartitions of $n$ into odd parts. The generating function for $\overline{p}_o(n)$ is given by
\begin{equation}\label{e3.3.3.1}
    \sum_{n=0}^{\infty}\overline{p}_o(n)q^n = \frac{(-q;q^2)_{\infty}}{(q;q^2)_{\infty}} = \frac{(q^2;q^2)_{\infty}^3}{(q;q)_{\infty}^2(q^4;q^4)_{\infty}}.
\end{equation}

In this section, we shall prove that the partition function $\overline{p}_o(n)$ can be expressed in terms of the partition function $pod(n)$.

\begin{theorem}
    For any nonnegative integer $n$,  we have
    \begin{equation*}
        pod(n) = \sum_{j=0}^{\infty}(-1)^jpod(j)\overline{p}_o(n-j).
    \end{equation*}
\end{theorem}

\begin{proof}
We can write
   \begin{equation*}
       \sum_{n=0}^{\infty}pod(n)q^n = \frac{(q;q)_{\infty}}{(q^2;q^2)_{\infty}^2}\sum_{n=0}^{\infty}\overline{p}_o(n)q^n,
    \end{equation*}
which is equivalent to
    \begin{equation*}
       \sum_{n=0}^{\infty}pod(n)q^n = \left( \sum_{n=0}^{\infty}(-1)^n pod(n)q^n \right) \left( \sum_{n=0}^{\infty}\overline{p}_o(n)q^n \right).
   \end{equation*}
Finally, equating the coefficient of $q^n$ on each side gives the result.
\end{proof}

\begin{theorem}
    For any nonnegative integer $n$, we have
    \begin{equation*}
        \overline{p}_o(n) = \sum_{j=0}^{\infty}pod\left(n-\frac{j(j+1)}{2}\right).
    \end{equation*}
\end{theorem}

\begin{proof}
Replacing $x$ by $1$ in $\eqref{e2.0.0}$ and $\eqref{e2.0.1}$, we have the relation
    \begin{equation*}
        (-q;q)_{\infty}(-1;q)_{\infty}(q;q)_{\infty} = \sum_{n=-\infty}^{\infty} q^{n(n+1)/2},
    \end{equation*}
which after simplification takes the shape
    \begin{equation*}
        2(-q;q)_{\infty}^2(q;q)_{\infty} = 2 \sum_{n=0}^{\infty} q^{n(n+1)/2},
    \end{equation*}
which is equivalent to
    \begin{equation*}
        \frac{(q;q^2)_{\infty}(q^2;q^2)_{\infty}}{(q;q^2)_{\infty}^2} = \sum_{n=0}^{\infty} q^{n(n+1)/2}.
    \end{equation*}
Further simplification of the above equation gives the identity
    \begin{equation}
        \frac{(q^2;q^2)_{\infty}^2}{(q;q)_{\infty}} = \sum_{n=0}^{\infty} q^{n(n+1)/2}. \label{e3.3.3.2}
    \end{equation}
Considering the identity $\eqref{e3.3.3.2}$ and the generating function of $\overline{p}_o(n)$, we can write
\begin{equation*}
    \sum_{n=0}^{\infty}\overline{p}_o(n)q^n = \left( \sum_{n=0}^{\infty}pod(n)q^n \right) \frac{(q^2;q^2)_{\infty}^2}{(q;q)_{\infty}},
\end{equation*}
which is equivalent to
\begin{equation*}
   \sum_{n=0}^{\infty}\overline{p}_o(n)q^n = \left( \sum_{n=0}^{\infty}pod(n)q^n \right) \left( \sum_{n=0}^{\infty} q^{n(n+1)/2} \right).
\end{equation*}
Finally, equating the coefficient of $q^n$ on each side gives the result.
\end{proof}

\subsection{Connection with the ordinary partition function and the partition function $q_{e-o}(n)$}
We denote the difference of the number of partitions of $n$ into distinct parts with an even number of odd parts and partitions of $n$ into distinct parts with an odd number of odd parts by $q_{e-o}(n)$. The generating function of $q_{e-o}(n)$ is given by
\begin{equation*}
    \sum_{n=0}^{\infty}q_{e-o}(n) q^n = \frac{1}{(-q;q^2)_{\infty}}.
\end{equation*}

\begin{theorem}
    For any nonnegative integer $n$, the partition functions $p(n)$, $q_{e-o}(n)$ and $pod(n)$ are related by
    \begin{equation*}
        p\left( \frac{n}{2} \right) = \sum_{j=0}^{\infty}pod(j)q_{\text{e$-$o}}(n-j),
    \end{equation*}
with $p(x)=0$ if $x$ is not an integer.
\end{theorem}

\begin{proof}
We can write
    \begin{equation*}
    \frac{1}{(-q;q^2)_{\infty}} \sum_{n=0}^{\infty}pod(n)q^n  = \frac{1}{(q^2;q^2)_{\infty}},
    \end{equation*}
which is equivalent to
    \begin{equation*}
    \left( \sum_{n=0}^{\infty}q_{e-o}(n) q^n \right) \left( \sum_{n=0}^{\infty}pod(n)q^n \right)  = \sum_{n=0}^{\infty}p(n)q^{2n}. 
\end{equation*}
Finally, equating the coefficient of $q^n$ on each side gives the result.
\end{proof}

\subsection{Connection with the partitions into parts congruent to $2\pmod{4}$}
We denote the number of partitions of $n$ into parts congruent to $2\pmod{4}$ by $P_2^{'}(n)$. It is clear that the generating function of $P_2^{'}(n)$ is given by
\begin{equation*}
    \sum_{n=0}^{\infty}P_2^{'}(n) = \frac{1}{(q^2;q^4)_{\infty}}. 
\end{equation*}

In this section, we will prove a relation connecting the partition functions $pod(n)$, $P_2^{'}(n)$ and $p(n)$.

\begin{theorem}
    For any nonnegative integer $n$, the partition functions $p(n)$, $P_2^{'}(n)$ and $pod(n)$ are related by
    \begin{equation*}
       p(n) = \sum_{j=0}^{\infty}pod(j)P_2^{'}(n-j).
    \end{equation*}
\end{theorem}

\begin{proof}
We can write
\begin{equation*}
   \frac{1}{(q^2;q^4)_{\infty}} \sum_{n=0}^{\infty}pod(n)q^n = \frac{1}{(q;q)_{\infty}}
\end{equation*}
which is equivalent to
\begin{equation*}
   \left( \sum_{n=0}^{\infty}pod(n) q^n \right)\left( \sum_{n=0}^{\infty}P_2^{'}(n)q^n \right) = \sum_{n=0}^{\infty}p(n)q^n.
\end{equation*}
Finally, equating the coefficient of $q^n$ on each side gives the result.
\end{proof}

\subsection{Connections with the ordinary partition function}

In this section, we will prove relations between the partition functions $p(n)$ and $pod(n)$.

\begin{theorem}
     For any nonnegative integer $n$, the partition functions $p(n)$ and $pod(n)$ are related by
    \begin{equation*}
       \sum_{j=-\infty}^{\infty}(-1)^jpod(n-3j^2) = \sum_{j=-\infty}^{\infty}p\left( \frac{n}{4}-\frac{j(3j+1)}{8} \right),
    \end{equation*}
with $p(x)=0$ if $x$ is not an integer.
\end{theorem}

\begin{proof}
Multiplying $\eqref{e2.1.3}$ by $(q;q)_{\infty}(q^3;q^3)_{\infty}/(q;q)_{\infty}(q^3;q^3)_{\infty}$, to obtain
\begin{equation*}
         \frac{(q^2;q^2)_{\infty}(q^3;q^3)_{\infty}^2}{(q;q)_{\infty}(q^6;q^6)_{\infty}} = \sum_{n=-\infty}^{\infty}q^{n(3n+1)/2}.
\end{equation*}
Multiplying the above equation by $1/(q^4;q^4)_{\infty}$, to get
\begin{equation*}
    \frac{(q^3;q^3)_{\infty}^2}{(q^6;q^6)_{\infty}}\frac{(q^2;q^2)_{\infty}}{(q;q)_{\infty}(q^4;q^4)_{\infty}}  = \frac{1}{(q^4;q^4)_{\infty}} \sum_{j=-\infty}^{\infty}q^{n(3n+1)/2},
\end{equation*}
which is equivalent to
\begin{equation*}
    \left( \sum_{n=-\infty}^{\infty}(-1)^nq^{3n^2} \right) \left( \sum_{n=0}^{\infty}pod(n)q^n \right)  = \left( \sum_{n=0}^{\infty} p(n)q^{4n} \right) \left( \sum_{n=-\infty}^{\infty} q^{n(3n+1)/2} \right).
\end{equation*}
Finally, equating the coefficient of $q^n$ on each side gives the result.
\end{proof}

\begin{theorem}
     For any nonnegative integer $n$, the partition functions $p(n)$ and $pod(n)$ are related by
    \begin{equation*}
       \sum_{j=0}^{\infty}pod(j)p\left( \frac{n}{2} - \frac{j}{2} \right) =  \sum_{j=0}^{\infty}p(j)p\left( \frac{n}{4} - \frac{j}{4} \right),
    \end{equation*}
with $p(x)=0$ if $x$ is not an integer.
\end{theorem}

\begin{proof}
We can write
\begin{equation*}
    \frac{1}{(q^2;q^2)_{\infty}} \sum_{n=0}^{\infty}pod(n)q^n = \frac{1}{(q;q)_{\infty}(q^4;q^4)_{\infty}}
\end{equation*}
which is equivalent to
\begin{equation*}
    \left( \sum_{n=0}^{\infty}pod(n)q^n \right)\left( \sum_{n=0}^{\infty}p(n)q^{2n} \right) = \left( \sum_{n=0}^{\infty}p(n)q^n \right)\left( \sum_{n=0}^{\infty}p(n)q^{4n} \right).
\end{equation*}
Finally, equating the coefficient of $q^n$ on each side gives the result.
\end{proof}

\subsection{Connections with the cubic partition function}
In $2010$, Chan \cite{7} introduced the cubic partition function $a(n)$ which counts the number of partitions of $n$ where the even parts can appear in two colors. For example, there are four cubic partitions of $3$, namely
\begin{equation*}
    3, \quad 2_1+1, \quad 2_2+1, \quad 1+1+1,
\end{equation*}
where the subscripts $1$ and $2$ denote the colors. The generating function of $a(n)$ satisfies the identity
\begin{equation*}
    \sum_{n=0}^{\infty}a(n)q^n = \frac{1}{(q;q)_{\infty}(q^2;q^2)_{\infty}}.
\end{equation*}
In this section, we shall prove that the partition function $pod(n)$ can be expressed in terms of the partition function $a(n)$.

\begin{theorem}
For any nonnegative integer $n$, the partition functions $a(n)$ and $pod(n)$ are related by
    \begin{equation}\label{e4.1.0.1}
      pod(n) = \sum_{j=-\infty}^{\infty} (-1)^ja(n-2j^2).
    \end{equation}
\end{theorem}

\begin{proof}
Replacing $q$ by $-q^2$ and $x$ by $-q^2$ in $\eqref{e0.1}$, we have
\begin{equation*}
    (q^2;q^4)_{\infty}(q^2;q^4)_{\infty}(q^4;q^4)_{\infty}  = \sum_{n=-\infty}^{\infty}(-1)^nq^{2n^2}.
\end{equation*}
Multiplying the above equation by $(q^4;q^4)_{\infty}/(q;q)_{\infty}$, to get
\begin{equation*}
    \frac{(q^2;q^2)_{\infty}^2}{(q;q)_{\infty}} = \frac{(q^4;q^4)_{\infty}}{(q;q)_{\infty}}\sum_{n=-\infty}^{\infty}(-1)^nq^{2n^2},
\end{equation*}
which after simplification becomes
\begin{equation*}
    \frac{(q^2;q^2)_{\infty}}{(q;q)_{\infty}(q^4;q^4)_{\infty}} = \frac{1}{(q;q)_{\infty}(q^2;q^2)_{\infty}}\sum_{n=-\infty}^{\infty}(-1)^nq^{2n^2},
\end{equation*}
which is equivalent to
\begin{equation*}
    \sum_{n=0}^{\infty}pod(n)q^n = \left( \sum_{n=0}^{\infty}a(n)q^n \right)\left( \sum_{n=-\infty}^{\infty}(-1)^nq^{2n^2} \right).
\end{equation*}
Finally, equating the coefficient of $q^n$ on each side gives the result.
\end{proof}

\begin{corollary}
For any nonnegative integer $n$, we have
    \begin{equation*}
        pod(n) \equiv a(n) \pmod{2}.
    \end{equation*}
\end{corollary}

\begin{proof}
    Equation $\eqref{e4.1.0.1}$ can be rewritten as
    \begin{equation*}
        pod(n)=a(n)+2\sum_{j=1}^{\infty}(-1)^{j}a(n-2j^2),
    \end{equation*}
which under modulo $2$, gives the result.
\end{proof}

\begin{theorem}
For any nonnegative integer $n$, the partition functions $a(n)$ and $pod(n)$ are related by triangular numbers as
    \begin{equation*}
      pod(n) = \sum_{j=0}^{\infty} a\left( \frac{n}{2}-\frac{j(j+1)}{4} \right),
    \end{equation*}
with $a(x)=0$ if $x$ is not an integer. 
\end{theorem}

\begin{proof}
    The generation function of $pod(n)$ can be written as
    \begin{equation*}
        \sum_{n=0}^{\infty}pod(n)q^n = \psi(q)\frac{1}{(q^2;q^2)_{\infty}(q^4;q^4)_{\infty}},
    \end{equation*}
which is equivalent to
    \begin{align*}
        \sum_{n=0}^{\infty}pod(n)q^n & = \left( \sum_{n=0}^{\infty}q^{T_n} \right) \left( \sum_{n=0}^{\infty}a(n)q^{2n} \right)
    \end{align*}
Finally, equating the coefficient of $q^n$ on each side gives the result.
\end{proof}

\subsection{Connection with the partition function $q_{odd}^{'''}(n)$}
The number of partitions of $n$ into distinct parts is usually denoted by $q(n)$. The number of partitions of $n$ into distinct odd parts is denoted in this paper by $q_{odd}(n)$. The generating functions for $q(n)$ and $q_{odd}(n)$ are given by
\begin{equation*}
    \sum_{n=0}^{\infty}q(n)q^n = (-q;q)_{\infty}
\end{equation*}
and
\begin{equation*}
    \sum_{n=0}^{\infty}q_{odd}(n)q^n = (-q;q^2)_{\infty}.
\end{equation*}
We denote in this paper, $q_{odd}^{'''}(n)$ to be the number of partitions of $n$ into distinct odd parts in $3$ colors. For example, we have $q_{odd}^{'''}(4)=9$, and the nine partitions are
\begin{equation*}
    3_1+1_1, \quad 3_2+1_2, \quad 3_3+1_3, \quad 3_1+1_2, \quad 3_1+1_3, \quad 3_2+1_1, \quad 3_3+1_1, \quad 3_2+1_3, \quad 3_3+1_2,
\end{equation*}
where the subscripts $1$, $2$ and $3$ denote the colors. It is clear that the generating function of $q_{odd}^{'''}(n)$ is given by
\begin{equation*}
    \sum_{n=0}^{\infty}q_{odd}^{'''}(n) = (-q;q^2)^3. 
\end{equation*}
In this section, we shall prove that the partition function $q_{odd}^{'''}(n)$ can be expressed in terms of the partition function $pod(n)$.

\begin{theorem}
    For any nonnegative integer $n$, the partition functions $q_{\text{odd}}^{'''}(n)$ and $pod(n)$ are related by
    \begin{equation}\label{e4.1.0.2}
      q_{\text{odd}}^{'''}(n) = \sum_{j=-\infty}^{\infty} pod(n-j^2).
    \end{equation}
\end{theorem}

\begin{proof}
Replacing $x$ by $1$ in $\eqref{e0.1}$, we have
\begin{equation*}
    (-q;q^2)^2_{\infty}(q^2;q^2)_{\infty} = \sum_{n=-\infty}^{\infty}q^{n^2}.
\end{equation*}
Multiplying the above equation by $(-q;q^2)_{\infty}/(q^2;q^2)_{\infty}$, to get
\begin{equation*}
    (-q;q^2)^3_{\infty} = \frac{(-q;q^2)_{\infty}}{(q^2;q^2)_{\infty}}\sum_{n=-\infty}^{\infty}q^{n^2},
\end{equation*}
which is equivalent to
\begin{equation*}
    \sum_{n=0}^{\infty}q_{odd}^{'''}(n) q^n = \left( \sum_{n=0}^{\infty}pod(n)q^n \right)\left( \sum_{n=-\infty}^{\infty} q^{n^2} \right).
\end{equation*}
Finally, equating the coefficient of $q^n$ on each side gives the result.
\end{proof}

\begin{corollary}
    For any nonnegative integer $n$, we have
    \begin{equation*}
        pod(n) \equiv q_{\text{odd}}^{'''}(n) \pmod{2}
    \end{equation*}
\end{corollary}
\begin{proof}
    Equation $\eqref{e4.1.0.2}$ can be rewritten as
    \begin{equation*}
        q_{\text{odd}}^{'''}(n)=pod(n)+2\sum_{j=1}^{\infty}pod(n-j^2),
    \end{equation*}
which under modulo $2$, gives the result.
\end{proof}

\subsection{Relation between the partition functions $pod(n)$, $p(n)$, $a(n)$ and $q_{odd}(n)$}
In this section, we shall prove a relation between the partition functions $pod(n)$, $p(n)$, $a(n)$ and $q_{odd}(n)$.
\begin{theorem}
    For any nonnegative integer $n$, the partition functions $pod(n)$, $p(n)$, $a(n)$ and $q_{\text{odd}}(n)$ are related by
    \begin{equation*}
     \sum_{j=0}^{\infty}pod(j)p(n-j) = \sum_{j=0}^{\infty}a(j)q_{\text{odd}}(n-j).
    \end{equation*}
\end{theorem}

\begin{proof}
We can write
\begin{equation*}
    \sum_{n=0}^{\infty}pod(n)q^n = \frac{(-q;q^2)_{\infty}}{(q^2;q^2)_{\infty}}\frac{(q;q)_{\infty}(q^2;q^2)_{\infty}}{(q;q)_{\infty}(q^2;q^2)_{\infty}},
\end{equation*}
which on simplification becomes
\begin{equation*}
    \frac{1}{(q;q)_{\infty}} \sum_{n=0}^{\infty}pod(n)q^n = (-q;q^2)_{\infty} \sum_{n=0}^{\infty}a(n)q^n ,
\end{equation*}
which is equivalent to
\begin{equation*}
    \left( \sum_{n=0}^{\infty}p(n)q^n \right)\left( \sum_{n=0}^{\infty}pod(n)q^n \right) = \left( \sum_{n=0}^{\infty} q_{odd}(n)q^n \right)\left( \sum_{n=0}^{\infty}a(n)q^n \right).
\end{equation*}
Finally, equating the coefficient of $q^n$ on each side gives the result.
\end{proof}

\subsection{Relation between the partition functions $p(n), \overline{\mathcal{E}\mathcal{O}}(n)$ and $pod(n)$}
Andrews \cite{6} introduced the partition function $\mathcal{E}\mathcal{O}(n)$ counts the number of partitions of $n$ where every even part is less than each odd part. For example, $\mathcal{E}\mathcal{O}(8)=12$ with the relevant partitions being 
\begin{equation*}
    8, \quad 6+2, \quad 7+1, \quad 4+4, \quad 4+2+2, \quad 5+3, \quad 5+1+1+1, \quad 2+2+2+2, \quad 3+3+2,
\end{equation*}
\begin{equation*}
    \quad 3+3+1+1, \quad 3+1+1+1+1+1, \quad  1+1+1+1+1+1+1+1.
\end{equation*}
The generating function for $\mathcal{E}\mathcal{O}(n)$ is
\begin{equation*}
    \sum_{n=0}^{\infty}\mathcal{E}\mathcal{O}(n)q^n = \frac{1}{(1-q)(q^2;q^2)_{\infty}}.
\end{equation*}
The partition function $\overline{\mathcal{E}\mathcal{O}}(n)$ is the number of partitions counted by $\mathcal{E}\mathcal{O}(n)$ in which only the largest even part appears an odd number of times. For example, $\overline{\mathcal{E}\mathcal{O}}(8)=5$, with the relevent partitions being
\begin{equation*}
    8, \quad 4+2+2, \quad 3+3+2, \quad 3+3+1+1, \quad 1+1+1+1+1+1+1+1.
\end{equation*}
The generating function of  $\overline{\mathcal{E}\mathcal{O}}(n)$ is 
\begin{equation*}
    \sum_{n=0}^{\infty}\overline{\mathcal{E}\mathcal{O}}(n)q^n = \frac{(q^4;q^4)_{\infty}}{(q^2;q^4)_{\infty}^2} = \frac{(q^4;q^4)_{\infty}^3}{(q^2;q^2)_{\infty}^2}.
\end{equation*}

In this section, we will prove a relation between the partition functions $p(n), \overline{\mathcal{E}\mathcal{O}}(n)$ and $pod(n)$.

\begin{theorem}
    For any nonnegative integer $n$, the partition functions $pod(n)$, $p(n)$, and $\overline{\mathcal{E}\mathcal{O}}(n)$ are related by
    \begin{equation*}
    \sum_{j=0}^{\infty}p(n-2T_j) = \sum_{j=0}^{\infty}pod(j)\overline{\mathcal{E}\mathcal{O}}(n-j).
    \end{equation*}
\end{theorem}

\begin{proof}
    We can write
    \begin{equation*}
        \sum_{n=0}^{\infty}\overline{\mathcal{E}\mathcal{O}}(n)q^n  \frac{(q^2;q^2)_{\infty}}{(q^4;q^4)_{\infty}} =\frac{(q^4;q^4)_{\infty}^2}{(q^2;q^2)_{\infty}}.
    \end{equation*}
Multiplying the above equation by $1/(q;q)_{\infty}$, to get
    \begin{equation*}
          \sum_{n=0}^{\infty}\overline{\mathcal{E}\mathcal{O}}(n)q^n  \frac{(q^2;q^2)_{\infty}}{(q^4;q^4)_{\infty} (q;q)_{\infty}}=\frac{(q^4;q^4)_{\infty}^2}{(q^2;q^2)_{\infty}(q;q)_{\infty}},
     \end{equation*}
which is equivalent to
    \begin{equation*}
        \left( \sum_{n=0}^{\infty}\overline{\mathcal{E}\mathcal{O}}(n)q^n \right) \left( \sum_{n=0}^{\infty}pod(n)q^n \right)  = \left( \sum_{n=0}^{\infty}p(n)q^n \right) \left( \sum_{n=0}^{\infty}q^{2T_n)} \right).
    \end{equation*}
Finally, equating the coefficient of $q^n$ on each side gives the result.    
\end{proof}

\subsection{Relation between the partition functions $pod(n)$ and $p_3(n)$}

In $2018$, Hirschhorn \cite{5} studied the number of partitions of $n$ in three colors, $p_3(n)$, given by the generating function
\begin{equation*}
    \sum_{n=0}^{\infty}p_3(n)q^n = \frac{1}{(q;q)_{\infty}^3}.
\end{equation*}

In this section, we shall prove a relation between partition function $pod(n)$ and partitions in three colors $p_3(n)$.

\begin{theorem}
     For any nonnegative integer $n$, the partition functions $pod(n)$ and $p_3(n)$ are related by
     \begin{equation}\label{e4.1.0.3}
         \sum_{j=0}^{\infty}pod(j)pod(n-j) = \sum_{j=-\infty}^{\infty}p_3\left( \frac{n}{2}-\frac{j^2}{2} \right),
     \end{equation}
with $p_3(x)=0$ if $x$ is not an integer. 
\end{theorem}

\begin{proof}
From $\eqref{e2.0.2}$, we have the relation
\begin{equation*}
\sum_{n=-\infty}^{\infty}q^{n^2}=\frac{(q^2;q^2)_{\infty}^5}{(q;q)_{\infty}^2(q^4;q^4)_{\infty}^2},
\end{equation*}
which can be written as
\begin{equation*}
    \sum_{n=-\infty}^{\infty}q^{n^2}  = \left( \frac{(q^2;q^2)_{\infty}}{(q;q)_{\infty}(q^4;q^4)_{\infty}} \right)^2 (q^2;q^2)_{\infty}^3.
\end{equation*}
Multiplying the above equation by $1/(q^2;q^2)_{\infty}^3$, to obtain
\begin{equation*}
    \frac{1}{(q^2;q^2)_{\infty}^3}\sum_{n=-\infty}^{\infty}q^{n^2}  = \left( \sum_{n=0}^{\infty}pod(n)q^n \right)^2
\end{equation*}
which is equivalent to
\begin{equation*}
    \left( \sum_{n=0}^{\infty}pod(n)q^n \right)^2  = \left( \sum_{n=0}^{\infty}p_3(n)q^{2n} \right) \left( \sum_{n=-\infty}^{\infty}q^{n^2} \right).
\end{equation*}
Finally, equating the coefficient of $q^n$ on each side gives the result.
\end{proof}

\begin{corollary}
    For any nonnegative integer $n$, we have
    \begin{equation*}
        p_3\left(\frac{n}{2}\right) \equiv \sum_{j=0}^{\infty}pod(j)pod(n-j) \pmod{2},
    \end{equation*}
    with $p_3(x)=0$ if $x$ is not an integer. 
\end{corollary}
\begin{proof}
    Equation $\eqref{e4.1.0.3}$ can be rewritten as
    \begin{equation*}
       \sum_{j=0}^{\infty}pod(j)pod(n-j)= p_3\left(\frac{n}{2}\right)  +  2\sum_{j=1}^{\infty}p_3\left( \frac{n}{2}-\frac{j^2}{2} \right).
    \end{equation*}
which under modulo $2$, gives the result.
\end{proof}

\section{Acknowledgement} 
The author wants to thank Prof. N.D. Baruah for taking the time to check this paper. His input and suggestions were really helpful in making this paper better.

\end{document}